\documentclass{amsart}
\usepackage[textwidth=144mm,textheight=214mm,centering]{geometry}
\usepackage{amssymb}
\usepackage[hyphens]{url}
\usepackage{amsmath}
\usepackage{amsthm}
\usepackage{mathrsfs}
\usepackage[title]{appendix}
\usepackage[utf8]{inputenc}
\usepackage[english]{babel}
\usepackage{cite}
\usepackage{aliascnt}
\usepackage{hyperref}
\usepackage{cleveref}
\usepackage[section]{placeins}
\RequirePackage{textgreek}
\usepackage{enumitem}
\usepackage{bm}
\usepackage{nicefrac}
\usepackage[all,cmtip]{xy}
\usepackage{caption}
\usepackage{float}
\usepackage{mathtools}
\usepackage{xcolor}
\usepackage{xspace}
\usepackage{stmaryrd}
\usepackage{dsfont}

\hypersetup{
    colorlinks=true,
    linkcolor=blue,
    citecolor=cyan,
    urlcolor=red
}

\hypersetup{linktoc=all}
\numberwithin{equation}{section}

\newtheorem{thm}{Theorem}[section]

\newaliascnt{lem}{thm}
\newtheorem{lem}[lem]{Lemma}
\aliascntresetthe{lem}
\newaliascnt{prop}{thm}
\newtheorem{prop}[prop]{Proposition}
\aliascntresetthe{prop}
\newaliascnt{cor}{thm}
\newtheorem{cor}[cor]{Corollary}
\aliascntresetthe{cor}
\newaliascnt{obs}{thm}

\aliascntresetthe{obs}
\newaliascnt{defn}{thm}
\newtheorem{defn}[defn]{Definition}
\aliascntresetthe{defn}
\newaliascnt{assump}{thm}

\aliascntresetthe{assump}

\newaliascnt{con}{thm}

\aliascntresetthe{con}
\newaliascnt{que}{thm}

\aliascntresetthe{que}
\newaliascnt{prob}{thm}

\aliascntresetthe{prob}
\newaliascnt{fct}{thm}

\aliascntresetthe{fct}

\newaliascnt{claim}{thm}

\aliascntresetthe{claim}
\theoremstyle{remark}
\newaliascnt{rem}{thm}
\newtheorem{rem}[rem]{Remark}
\aliascntresetthe{rem}
\newaliascnt{exm}{thm}

\aliascntresetthe{exm}
\newaliascnt{exs}{thm}

\aliascntresetthe{exs}

\crefname{thm}{theorem}{theorems}
\Crefname{thm}{Theorem}{Theorems}
\crefname{lem}{lemma}{lemmas}
\Crefname{lem}{Lemma}{Lemmas}
\crefname{prop}{proposition}{propositions}
\Crefname{prop}{Proposition}{Propositions}
\crefname{cor}{corollary}{corollaries}
\Crefname{cor}{Corollary}{Corollaries}
\crefname{obs}{observation}{observations}
\Crefname{obs}{Observation}{Observations}
\crefname{defn}{definition}{definitions}
\Crefname{defn}{Definition}{Definitions}
\crefname{que}{question}{questions}
\Crefname{que}{Question}{Questions}
\crefname{prob}{problem}{problems}
\Crefname{prob}{Problem}{Problems}
\crefname{fct}{fact}{facts}
\Crefname{fct}{Fact}{Facts}
\crefname{claim}{claim}{claims}
\Crefname{claim}{Claim}{Claims}
\crefname{rem}{remark}{remarks}
\Crefname{rem}{Remark}{Remarks}
\crefname{exm}{example}{examples}
\Crefname{exm}{Example}{Examples}
\crefname{exs}{examples}{examples}
\Crefname{exs}{Examples}{Examples}

\makeatletter
\def\@settitle{%
  \begin{center}%
    \baselineskip18\p@\relax
    \normalfont\Large\scshape
    \@title
  \end{center}%
}
\makeatother

\author{Nachi Avraham-Re'em}
\address[Nachi Avraham-Re'em]{Department of Mathematics, Technion -- Israel Institute of Technology, Haifa, Israel}
\email{nachi.avraham@gmail.com}

\author{Zemer Kosloff}
\address[Zemer Kosloff]{School of Mathematics, University of Bristol, Fry Building,
Woodland Road, Bristol BS8 1UG, United Kingdom\newline
Einstein Institute of Mathematics, The Hebrew University of Jerusalem,
Edmond J. Safra Campus, Givat Ram, Jerusalem 9190401, Israel}
\email{zemer.kosloff@bristol.ac.uk}
\email{zemer.kosloff@mail.huji.ac.il}

\thanks{Z. Kosloff was supported by the ISF (grant No. 1180/22). N. Avraham-Re'em was supported by the ISF (grant No. 3423/24) and by an Aly Kaufman Fellowship at the Technion.}

\title{Asymmetry of \(\ell^{2}\)-cohomology via skewed F{\o}lner geometry}

\subjclass[2020]{Primary 20F65, 20F69, 20J06; Secondary 37A40, 43A07}

\keywords{\(\ell^{2}\)-cohomology, Dirichlet spaces, asymmetric cocycles, amenable groups, FC-groups, F{\o}lner geometry, nonsingular Bernoulli shifts}

\begin{document}

\begin{abstract}
We study the two \(\ell^{2}\)-Dirichlet structures on a countable group \(G\) arising from the left and right regular actions on \(\mathbb{R}^{G}\). Although the left and right regular representations on \(\ell^{2}\left(G\right)\) are unitarily equivalent, the \(\ell^{2}\)-Dirichlet subspaces of \(\mathbb{R}^{G}\) need not coincide. Our main result gives a complete classification of this asymmetry for countable amenable groups:
\[\mathcal{D}_{2}\left(G,\lambda\right)=\mathcal{D}_{2}\left(G,\rho\right)\quad\Longleftrightarrow\quad G \text{ is an FC-group}.\]

The proof is based on a skewed F{\o}lner-geometric mechanism, called a \emph{left scheme}, combining summability of left boundaries with displacement under a right translation. We develop this mechanism generally, and demonstrate it concretely in the Heisenberg group and amenable wreath products over \(\mathbb{Z}\).

We also show that this mechanism has a dynamical counterpart in the theory of nonsingular Bernoulli shifts: every countable amenable group that is not an FC-group admits Bernoulli schemes whose left shift is nonsingular, conservative and weakly mixing, whereas the right shift by some element is singular.
\end{abstract}

\maketitle

\tableofcontents

\section{Introduction and results}

A theorem of Cheeger--Gromov~\cite{cheeger1986} asserts that for countable amenable groups, the reduced \(\ell^{2}\)-cohomology vanishes; see \cite[Cor.~5.13]{luck1998}, \cite{kyed2015}. Thus the phenomena studied in this paper are invisible from the reduced point of view, and in particular from \(\ell^{2}\)-Betti numbers. The central question we deal with is whether the two canonical Dirichlet structures induced by left and right translation agree as subspaces of \(\mathbb{R}^{G}\). We shall explain why in countable nonamenable groups this has a striking answer, and then give a full classification in the amenable case: for a countable amenable group \(G\), the left and right \(\ell^{2}\)-Dirichlet spaces differ precisely when \(G\) is not an FC-group.

\smallskip

A central object in first \(\ell^{2}\)-cohomology of groups is the \(\ell^{2}\)-Dirichlet space. For a countable group \(G\), this is the space of real-valued functions defining \(\ell^{2}\left(G\right)\)-valued \(1\)-cocycles for the regular representation. Beyond its intrinsic interest in geometric group theory and representation theory (see e.g.~\cite[\S8]{gromov1993}, \cite{bekka1997,martin2007,gournay2018}), this space lies at the heart of the ergodic theory of nonsingular Bernoulli shifts. Indeed, by Kakutani's classical \(\ell^{2}\)-type criterion for equivalence of product measures \cite{kakutani1948}, the nonsingularity of Bernoulli actions is governed precisely by whether the square-root marginal probabilities belong to the appropriate \(\ell^{2}\)-Dirichlet space; see~\cite{vaes2018,bjorklund2018,bjorklund2021}.

\smallskip

We now define these two Dirichlet structures precisely. Fix a countably infinite discrete group \(G\). For brevity, throughout we refer to such \(G\) simply as a countable group. Denote by \(\lambda\) and \(\rho\) the left and right translation actions of \(G\) on the vector space \(\mathbb{R}^{G}\), defined by
\[\lambda\left(g\right)\xi\left(h\right)=\xi\left(g^{-1}h\right)\quad\text{and}\quad \rho\left(g\right)\xi\left(h\right)=\xi\left(hg\right).\]
Define the {\bf left \(\ell^{2}\)-Dirichlet space} of \(G\) by
\[\mathcal{D}_{2}\left(G,\lambda\right)\coloneq\left\{\xi\in\mathbb{R}^{G}:\forall g\in G,\ \lambda\left(g\right)\xi-\xi\in\ell^{2}\left(G\right)\right\},\]
and the {\bf right \(\ell^{2}\)-Dirichlet space} of \(G\) by
\[\mathcal{D}_{2}\left(G,\rho\right)\coloneq\left\{\xi\in\mathbb{R}^{G}:\forall g\in G,\ \rho\left(g\right)\xi-\xi\in\ell^{2}\left(G\right)\right\}.\]
Inside \(\mathcal{D}_{2}\left(G,\lambda\right)\cap \mathcal{D}_{2}\left(G,\rho\right)\) one always finds the trivial subspace
\[\ell^{2}\left(G\right)\oplus\mathbb{R}\mathbf{1}.\]
We will use the following slightly abused but convenient terminology:
\begin{itemize}
    \item the elements of \(\mathcal{D}_{2}\left(G,\lambda\right)\) are called {\bf left \(\ell^{2}\)-cocycles};
    \item the elements of \(\mathcal{D}_{2}\left(G,\rho\right)\) are called {\bf right \(\ell^{2}\)-cocycles};
    \item the elements of \(\mathcal{D}_{2}\left(G,\lambda\right)\cap \mathcal{D}_{2}\left(G,\rho\right)\) are called {\bf symmetric \(\ell^{2}\)-cocycles};
    \item the elements of \(\mathcal{D}_{2}\left(G,\lambda\right)\triangle \mathcal{D}_{2}\left(G,\rho\right)\) are called {\bf asymmetric \(\ell^{2}\)-cocycles};
    \item the elements of \(\ell^{2}\left(G\right)\oplus\mathbb{R}\mathbf{1}\) are called {\bf \(\ell^{2}\)-coboundaries}, which are symmetric \(\ell^{2}\)-cocycles.\footnote{The terminology is justified by the fact that \(\xi\mapsto b_{\xi}\), \(b_{\xi}\left(g\right)\coloneq\lambda\left(g\right)\xi-\xi\), identifies \(Z^{1}\left(G,\lambda\right)\cong\mathcal{D}_{2}\left(G,\lambda\right)/\mathbb{R}\mathbf{1}\) and \(B^{1}\left(G,\lambda\right)\cong\left(\ell^{2}\left(G\right)\oplus\mathbb{R}\mathbf{1}\right)/\mathbb{R}\mathbf{1}\), and similarly for \(\rho\).}
\end{itemize}

\smallskip

Although the left and right regular representations are unitarily equivalent as abstract unitary representations \(G\to\mathrm{U}\left(\ell^{2}\left(G\right)\right)\) (via variable inversion), their \(\ell^{2}\)-Dirichlet spaces need not agree as subspaces of \(\mathbb{R}^{G}\). This discrepancy is not accidental: the unreduced \(\ell^{2}\)-cohomology can retain directional structure.

\smallskip

Two phenomena shape the general picture of \(\ell^{2}\)-asymmetry. The first is an algebraic obstruction. Abelian groups, and more generally FC-groups, admit no asymmetric \(\ell^{2}\)-cocycles (\cref{prop:FC}). Thus, for FC-groups,
\[\mathcal{D}_{2}\left(G,\lambda\right)=\mathcal{D}_{2}\left(G,\rho\right).\]

\smallskip

The second phenomenon is of geometric or potential-theoretic nature: for nonamenable groups, the space of \(\ell^{2}\)-cocycles has a particularly restrictive form~\cite{bekka1997}. This is manifested in the following theorem. One of its implications is due to Stefaan Vaes, who kindly allowed us to include his proof; see Appendix~\ref{app:vaes}.

\begin{thm}\label{mthm1}
For every countably infinite group \(G\),
\[\mathcal{D}_{2}\left(G,\lambda\right)\cap\mathcal{D}_{2}\left(G,\rho\right)=\ell^{2}\left(G\right)\oplus\mathbb{R}\mathbf{1}\iff G\text{ is nonamenable.}\]
\end{thm}

Among nonamenable groups, some exhibit additional cohomological obstructions. For instance, for groups with Kazhdan's property (T), such as \(\mathrm{SL}_{3}\left(\mathbb{Z}\right)\), all \(\ell^{2}\)-cocycles are \(\ell^{2}\)-coboundaries by the Delorme--Guichardet theorem~\cite[Prop.~2.2.10]{bekka2008}, and therefore they admit no asymmetric \(\ell^{2}\)-cocycle. Other nonamenable groups, such as nonabelian free groups, have positive first \(\ell^{2}\)-Betti number, and therefore they admit asymmetric \(\ell^{2}\)-cocycles that are not \(\ell^{2}\)-coboundaries. In light of \cref{mthm1}, this work is concerned with the following question.
\begin{quote}
\emph{Which countable amenable groups admit asymmetric \(\ell^{2}\)-cocycles?}
\end{quote}

This question was raised in a MathOverflow discussion~\cite{MOthread}, where it was noted that amenable groups appear to be the interesting case; this is now substantiated by \cref{mthm1}. In the same discussion, M. Kapovich observed that the infinite dihedral group \(D_{\infty}\) admits an asymmetric \(\ell^{2}\)-cocycle. We give a complete classification that settles this question:

\begin{thm}\label{mthm2}
For every countable amenable group \(G\),
\[\mathcal{D}_{2}\left(G,\lambda\right)=\mathcal{D}_{2}\left(G,\rho\right)\iff G\text{ is an FC-group}.\]
\end{thm}

The usual Dirichlet topology on \(\mathcal{D}_{2}\left(G,\lambda\right)\) is pointwise convergence of the corresponding \(\ell^{2}\)-valued cocycles. More precisely, for \(\xi\in\mathcal{D}_{2}\left(G,\lambda\right)\) let \(b_{\xi}:G\to\ell^{2}\left(G\right)\) be the cocycle \(b_{\xi}\left(g\right)=\lambda\left(g\right)\xi-\xi\). Then \(\xi_{n}\to\xi\) if \(b_{\xi_{n}}\to b_{\xi}\) pointwise, meaning that
\[\left\Vert\lambda\left(g\right)\left(\xi_{n}-\xi\right)-\left(\xi_{n}-\xi\right)\right\Vert_{\ell^{2}\left(G\right)}\to0\quad\text{for every }g\in G.\]
To make it Hausdorff, one may also require \(\xi_{i}\left(e_{G}\right)\to\xi\left(e_{G}\right)\) (the density assertion below is valid in either convention). The aforementioned Cheeger--Gromov theorem says that for amenable \(G\), the space \(\ell^{2}\left(G\right)\oplus\mathbb{R}\mathbf{1}\) is dense in \(\mathcal{D}_{2}\left(G,\lambda\right)\), hence so is the larger set of symmetric \(\ell^{2}\)-cocycles \(\mathcal{D}_{2}\left(G,\lambda\right)\cap\mathcal{D}_{2}\left(G,\rho\right)\). We observe that in the non-FC case, the complement of this set inside \(\mathcal{D}_{2}\left(G,\lambda\right)\) is dense as well.

\begin{cor}\label{cor:dense}
For every countable amenable group \(G\) that is not an FC-group, the set of left asymmetric \(\ell^{2}\)-cocycles, \(\mathcal{D}_{2}\left(G,\lambda\right)\setminus\mathcal{D}_{2}\left(G,\rho\right)\), is dense in \(\mathcal{D}_{2}\left(G,\lambda\right)\).
\end{cor}

The proof of \cref{mthm2} introduces the main technical object of this work, which we call \emph{left scheme}. Informally, a left scheme is a sequence of finite subsets with two competing features: summability of normalized left boundaries, and strong displacement under right translation by a distinguished element. Thus a left scheme is a skewed F{\o}lner configuration: the sets are almost invariant from the left in a summable sense, but are separated by some right translation.

\smallskip

Recall that the first (unreduced) \(\ell^{2}\)-cohomology of the left regular representation of \(G\) is
\[H^{1}\left(G,\lambda\right)\coloneq Z^{1}\left(G,\lambda\right)\big/B^{1}\left(G,\lambda\right),\]
where \(Z^{1}\left(G,\lambda\right)\) is the space of \(\ell^{2}\)-cocycles for \(\lambda\), and \(B^{1}\left(G,\lambda\right)\) is the subspace of \(\ell^{2}\)-coboundaries. Constructing nonzero classes in \(H^{1}\left(G,\lambda\right)\) is generally challenging. Since asymmetric \(\ell^{2}\)-cocycles are never \(\ell^{2}\)-coboundaries, they are a source of nonzero classes. Moreover, since they are never \(\ell^{2}\)-coboundaries, the associated cocycles \(b_{\xi}\) are non-inner, hence unbounded, so for amenable groups they are proper by the dichotomy of Peterson--Thom~\cite[Thm.~2.5]{peterson2011}.

\subsection{Asymmetric Bernoulli schemes}\label{sct:asymBer}

We next explain the dynamical meaning of this \(\ell^{2}\)-asymmetry. The relevant class of actions is nonsingular Bernoulli shifts, which form a fundamental and flexible family in ergodic theory: they are the nonsingular analogue of the classical probability preserving Bernoulli shifts, obtained by allowing the marginal probabilities to vary with the group element. They have attracted substantial recent attention in ergodic theory, operator algebras, and measurable group theory; see~\cite{vaes2018,bjorklund2018,kosloff2019,danilenko2019,bjorklund2021,Kosloff2021,berendschot2022} and the survey \cite{danilenko2023}.

\smallskip

For the purposes of the present paper, nonsingular Bernoulli shifts provide a canonical setting in which the distinction between left and right \(\ell^{2}\)-Dirichlet structures becomes dynamically visible. In the probability preserving Bernoulli case the marginals are constant, so the left and right shifts preserve the same product measure. In the nonsingular case, however, the vector of marginal probabilities has a nontrivial orbit when translated by group elements. Kakutani's criterion detects precisely whether the measure class is compatible with this translation action. Thus the question of whether \(\mathcal{D}_{2}\left(G,\lambda\right)\) and \(\mathcal{D}_{2}\left(G,\rho\right)\) agree has a direct dynamical counterpart: it asks whether the same marginal data that make a Bernoulli product measure nonsingular for the left shift also make it nonsingular for the right shift.

\smallskip

Let \(G\) be a countable group, and consider the product space
\[\Omega_{G}\coloneq\{0,1\}^{G}.\]
This space has two canonical actions of \(G\), namely the left shift and the right shift:
\[L\left(g\right)\left(\omega_{h}\right)_{h\in G}=\left(\omega_{g^{-1}h}\right)_{h\in G}\quad\text{and}\quad R\left(g\right)\left(\omega_{h}\right)_{h\in G}=\left(\omega_{hg}\right)_{h\in G}.\]
For a \(\left(0,1\right)\)-valued vector \(\xi\in\left(0,1\right)^{G}\subset\mathbb{R}^{G}\), define a probability product measure \(\mu_{\xi}\) on \(\Omega_{G}\) by
\[\mu_{\xi}\coloneq\bigotimes\nolimits_{h\in G}\left(\xi\left(h\right),1-\xi\left(h\right)\right).\]
The probability space \(\left(\Omega_{G},\mu_{\xi}\right)\) is called a {\bf Bernoulli scheme}. When \(\xi\) is constant, the measure \(\mu_{\xi}\) is invariant under both the left and the right shifts, and one recovers the usual probability preserving Bernoulli shift. Thus the classical Bernoulli shifts sit inside this family as the most symmetric case.

\smallskip

For a nonconstant \(\xi\), the measure \(\mu_{\xi}\) is no longer invariant, and the natural replacement for invariance is nonsingularity (also known as quasi-invariance): the shift should preserve the measure class of \(\mu_{\xi}\). The Bernoulli scheme is {\bf left nonsingular} if \(L\left(g\right)_{\ast}\mu_{\xi}\sim\mu_{\xi}\) for every \(g\in G\), and {\bf right nonsingular} if \(R\left(g\right)_{\ast}\mu_{\xi}\sim\mu_{\xi}\) for every \(g\in G\). Since
\[L\left(g\right)_{\ast}\mu_{\xi}=\mu_{\lambda\left(g\right)\xi}\quad\text{and}\quad R\left(g\right)_{\ast}\mu_{\xi}=\mu_{\rho\left(g\right)\xi},\]
the nonsingularity of the shift is a question on absolute continuity of product measures. Then by Kakutani's theorem~\cite{kakutani1948}, the Bernoulli scheme \(\left(\Omega_{G},\mu_{\xi}\right)\) is left nonsingular precisely when
\[\sqrt{\xi}\in\mathcal{D}_{2}\left(G,\lambda\right)\quad\text{and}\quad\sqrt{\mathbf{1}-\xi}\in \mathcal{D}_{2}\left(G,\lambda\right).\]
The analogous criterion with \(\rho\) characterizes right nonsingularity.

\smallskip

Nonsingularity alone is a measure class condition. From a dynamical point of view, the next basic requirement is conservativity, meaning measurable recurrence; see \cref{sct:nonsBer}. For probability preserving actions conservativity follows from the Poincar\'{e} recurrence theorem, but for nonsingular actions it is an additional and often delicate property. In the present Bernoulli setting, conservativity is especially useful because it automatically upgrades to a strong ergodic behavior. Indeed, a conservative nonsingular Bernoulli shift over an amenable group whose marginals stay away from \(0\) and \(1\) is weakly mixing. This was proved by one of the authors for \(G=\mathbb{Z}\) with ergodicity \cite[Thm.~3]{kosloff2019}, and generalized to weak mixing for countable amenable groups by Danilenko~\cite[Thm.~0.2]{danilenko2019} (cf. \cite[\S5]{bjorklund2021}). Therefore, to obtain dynamically meaningful Bernoulli schemes, one must also control the conservativity of the resulting shift action.

\smallskip

The left schemes constructed for \cref{mthm2} already produce marginal vectors which are compatible with the left shift and incompatible with the right shift. To make the resulting Bernoulli shifts conservative, we refine these schemes by imposing a recurrence condition. This recurrence condition controls the growth of the Radon--Nikodym cocycle, and allows us to apply a criterion of Vaes--Wahl~\cite[\S4]{vaes2018} (see \cref{prop:VaWa}). The result is that the same algebraic obstruction, namely the failure of the FC property, exactly produces asymmetric nonsingular Bernoulli shifts with strong ergodic behavior.

\begin{thm}\label{mthm3}
Every countable amenable group \(G\) that is not an FC-group admits a Bernoulli scheme \(\left(\Omega_{G},\mu_{\xi}\right)\) whose left shift action is nonsingular, conservative and weakly mixing, while the right shift by some element of \(G\) is singular.
\end{thm}

\begin{rem}\label{rem:partcons}
When \(G\) admits an element \(s_{\ast}\) which has infinite order and infinite conjugacy class, one can construct the Bernoulli scheme \(\left(\Omega_{G},\mu_{\xi}\right)\) in such a way that, additionally to the properties in \cref{mthm3}, the restricted left \(\mathbb{Z}\)-action of \(\left\langle s_{\ast}\right\rangle\) is nonsingular and conservative, while the right shift by \(s_{\ast}\) is singular.
\end{rem}

\subsection*{Organization}

In \cref{sct:lftsch} we introduce left schemes and prove that a countable amenable group admits a left scheme precisely when it is not an FC-group. We then show that left schemes produce asymmetric \(\ell^{2}\)-cocycles. In \cref{sct:reclftsch} we introduce recurrent left schemes and use them to construct conservative asymmetric Bernoulli schemes. We then prove that every countable amenable non-FC group admits recurrent left schemes. In \cref{sct:geomleftsch} we give concrete geometric constructions in the discrete Heisenberg group and in amenable wreath products \(A\wr\mathbb{Z}\). Finally, in \cref{sct:finproofs} we assemble the proofs of the main theorems. Appendix~\ref{app:vaes} presents Vaes' proof of the nonamenable implication in \cref{mthm1}.

\section{Left schemes and asymmetric \texorpdfstring{\(\ell^{2}\)}{l2}-cocycles}\label{sct:lftsch}

The purpose of this section is to introduce left schemes and explain how they are used in asymmetric constructions. We will start with a general construction of symmetric \(\ell^{2}\)-cocycles.

\subsection{Balanced F{\o}lner geometry}

The following construction of symmetric \(\ell^{2}\)-cocycles uses balanced (two-sided) F{\o}lner geometry, and it will guide the later asymmetric construction.

\begin{prop}\label{prop:symmetcoc}
Every countably infinite amenable group \(G\) admits a symmetric \(\ell^{2}\)-cocycle which is not an \(\ell^{2}\)-coboundary.
\end{prop}

We first record an elementary lemma using the F{\o}lner property.

\begin{lem}\label{lem:balanced}
Let \(G\) be a countably infinite amenable group with an exhaustion by finite symmetric sets \(K_{1}\subseteq K_{2}\subseteq\dotsm\). There is a sequence \(\left(F_{n}\right)_{n\geq 1}\) of pairwise disjoint finite sets with \(\left|F_{n}\right|\to\infty\), such that 
\[\max_{g\in K_{n}}\frac{\left|gF_{n}\triangle F_{n}\right|}{\left|F_{n}\right|}\leq\frac{1}{n^{2}}\quad\text{and}\quad\max_{g\in K_{n}}\frac{\left|F_{n}g\triangle F_{n}\right|}{\left|F_{n}\right|}\leq\frac{1}{n^{2}}\quad\text{for every }n\geq 1.\]
\end{lem}

\begin{proof}[Proof of \cref{lem:balanced}]
Since \(G\) is amenable, it admits a two-sided F{\o}lner sequence (see e.g. \cite[Thm.~4.10]{kerrli2016}). Fix such a F{\o}lner sequence \(\left(S_{m}\right)_{m\geq1}\) and define \(\left(F_{n}\right)_{n\geq1}\) inductively as follows. Define \(F_{0}\coloneq\emptyset\) and put \(E_{1}\coloneq\emptyset\). Assume \(F_{0},\dotsc,F_{n-1}\) are defined and put \(E_{n}\coloneq F_{1}\cup\dotsm\cup F_{n-1}\). Choose \(m>n\) large so that
\[\max_{g\in K_{n}}\frac{\left|gS_{m}\triangle S_{m}\right|}{\left|S_{m}\right|}<\frac{1}{4n^{2}},\quad\max_{g\in K_{n}}\frac{\left|S_{m}g\triangle S_{m}\right|}{\left|S_{m}\right|}<\frac{1}{4n^{2}},\quad\frac{\left|S_{m}\cap E_{n}\right|}{\left|S_{m}\right|}<\frac{1}{4n^{2}},\quad\left|S_{m}\right|>n,\]
where the last two conditions hold because \(\left|S_{m}\right|\to\infty\). Define then \(F_{n}\coloneq S_{m}\setminus E_{n}\).

\smallskip

The sets \(\left(F_{n}\right)_{n\geq1}\) are pairwise disjoint by construction, and for every \(n\) we have
\[\left|F_{n}\right|=\left|S_{m}\right|-\left|S_{m}\cap E_{n}\right|>\left(1-1/4n^{2}\right)\cdot\left|S_{m}\right|\geq 3/4\cdot\left|S_{m}\right|.\]
In particular, since \(\left|S_{m}\right|>n\), this shows that \(\left|F_{n}\right|\to\infty\). We show the left boundary condition, and the right boundary condition is similar. Fix \(g\in K_{n}\). Then
\[gF_{n}\triangle F_{n}\subseteq\left(gS_{m}\triangle S_{m}\right)\cup g\left(S_{m}\cap E_{n}\right)\cup\left(S_{m}\cap E_{n}\right),\]
we get
\[\left|gF_{n}\triangle F_{n}\right|\leq\left|gS_{m}\triangle S_{m}\right|+2\left|S_{m}\cap E_{n}\right|<\left(1/4n^{2}+2/4n^{2}\right)\cdot\left|S_{m}\right|=3/4n^{2}\cdot\left|S_{m}\right|.\]
Then both estimates combined give the left boundary condition.
\end{proof}

\begin{proof}[Proof of \cref{prop:symmetcoc}]
Fix an exhaustion \(K_{1}\subseteq K_{2}\subseteq\dotsm\) of \(G\) by finite symmetric sets, and let \(\left(F_{n}\right)_{n\geq1}\) be the sequence supplied by \cref{lem:balanced}. Define
\[\xi\in\mathbb{R}^{G},\quad\xi\left(h\right)\coloneq\sum\nolimits_{n\geq1}\frac{1}{\sqrt{\left|F_{n}\right|}}\cdot\mathbf{1}_{F_{n}}\left(h\right),\quad h\in G.\]
Fix \(g\in G\), and put
\[n\left(g\right)\coloneq\min\left\{n\geq1:g\in K_{n}\right\}.\]
Since the sets \(\left(F_{n}\right)_{n\geq1}\) are pairwise disjoint, so are the sets \(\left(gF_{n}\right)_{n\geq1}\). Therefore, for every \(h\in G\), at most one of the terms \(\mathbf{1}_{gF_{n}}\left(h\right)\) is nonzero, and at most one of the terms \(\mathbf{1}_{F_{n}}\left(h\right)\) is nonzero. Then we get
\[\Big|\sum\nolimits_{n\geq n\left(g\right)}\frac{1}{\sqrt{\left|F_{n}\right|}}\left(\mathbf{1}_{gF_{n}}\left(h\right)-\mathbf{1}_{F_{n}}\left(h\right)\right)\Big|^{2}\leq\sum\nolimits_{n\geq n\left(g\right)}\frac{1}{\left|F_{n}\right|}\mathbf{1}_{gF_{n}\triangle F_{n}}\left(h\right).\]
Summing over \(h\in G\), and using that \(g\in K_{n}\) for every \(n\geq n\left(g\right)\), we obtain
\[\Big\Vert\sum\nolimits_{n\geq n\left(g\right)}\frac{1}{\sqrt{\left|F_{n}\right|}}\left(\mathbf{1}_{gF_{n}}-\mathbf{1}_{F_{n}}\right)\Big\Vert_{\ell^{2}\left(G\right)}^{2}\leq\sum\nolimits_{n\geq n\left(g\right)}\frac{\left|gF_{n}\triangle F_{n}\right|}{\left|F_{n}\right|}\leq\sum\nolimits_{n\geq n\left(g\right)}\frac{1}{n^{2}}<+\infty.\]
Since there are only finitely many terms with \(n<n\left(g\right)\), it follows that \(\lambda\left(g\right)\xi-\xi\in\ell^{2}\left(G\right)\).

For the right action, one repeats the same argument with \(\left(F_{n}g^{-1}\right)_{n\geq1}\) in place of \(\left(gF_{n}\right)_{n\geq1}\). Since \(K_{n}\) is symmetric and \(g\in K_{n}\) for every \(n\geq n\left(g\right)\), we also have \(g^{-1}\in K_{n}\) for every \(n\geq n\left(g\right)\), and therefore
\[\frac{\left|F_{n}g^{-1}\triangle F_{n}\right|}{\left|F_{n}\right|}\leq\frac{1}{n^{2}}\quad\text{for every }n\geq n\left(g\right).\]
Thus \(\rho\left(g\right)\xi-\xi\in\ell^{2}\left(G\right)\). As \(g\in G\) is arbitrary, it follows that \(\xi\in\mathcal{D}_{2}\left(G,\lambda\right)\cap\mathcal{D}_{2}\left(G,\rho\right)\).

\smallskip

Let us show that \(\xi\notin\ell^{2}\left(G\right)\oplus\mathbb{R}\mathbf{1}\). Assume otherwise, so that \(\xi-p\in\ell^{2}\left(G\right)\) for some \(p\in\mathbb{R}\). If \(p=0\), since \(\left(F_{n}\right)_{n\geq1}\) are pairwise disjoint we get
\[\left\Vert\xi\right\Vert_{\ell^{2}\left(G\right)}^{2}=\sum\nolimits_{n\geq1}\frac{1}{\left|F_{n}\right|}\cdot\left|F_{n}\right|=\sum\nolimits_{n\geq1}1=+\infty,\]
contrary to the assumption. If \(p\neq 0\), since \(\left|F_{n}\right|\to\infty\) we may pick \(N\geq1\) such that \(\big|\frac{1}{\sqrt{\left|F_{n}\right|}}-p\big|>\left|p\right|/2\) for all \(n\geq N\). Then using that \(\xi\mid_{F_{n}}\equiv\frac{1}{\sqrt{\left|F_{n}\right|}}\) we obtain
\[\left\Vert\xi-p\right\Vert_{\ell^{2}\left(G\right)}^{2}\geq\sum\nolimits_{n\geq1}\sum\nolimits_{h\in F_{n}}\Big|\frac{1}{\sqrt{\left|F_{n}\right|}}-p\Big|^{2}\geq\left|p\right|^{2}/4\cdot\sum\nolimits_{n\geq N}\left|F_{n}\right|=+\infty,\]
which is again contrary to the assumption. Hence \(\xi\notin\ell^{2}\left(G\right)\oplus\mathbb{R}\mathbf{1}\).
\end{proof}

\subsection{FC-groups}

A group \(G\) is called an {\bf FC-group} if every element of \(G\) has finite conjugacy class. Let us see that the commutativity obstruction to \(\ell^{2}\)-asymmetry generalizes to the class of FC-groups.

\begin{prop}\label{prop:FC}
If \(G\) is an FC-group then \(\mathcal{D}_{2}\left(G,\lambda\right)=\mathcal{D}_{2}\left(G,\rho\right)\).
\end{prop}

\begin{proof}[Proof of \cref{prop:FC}]
For \(g,h\in G\) let \(g^{h}\coloneq hgh^{-1}\). For every \(\xi\in\mathbb{R}^{G}\) and \(g,h\in G\) we have
\[\left(\rho\left(g^{-1}\right)\xi-\xi\right)\left(h\right)=\xi\left(hg^{-1}\right)-\xi\left(h\right)=\left(\lambda\left(g^{h}\right)\xi-\xi\right)\left(h\right),\]
and therefore
\[\left\Vert\rho\left(g^{-1}\right)\xi-\xi\right\Vert_{\ell^{2}\left(G\right)}^{2}\leq\sum\nolimits_{c\in g^{G}}\left\Vert\lambda\left(c\right)\xi-\xi\right\Vert_{\ell^{2}\left(G\right)}^{2}.\]
Now if \(G\) is an FC-group then \(g^{G}\) is finite, and hence \(\xi\in\mathcal{D}_{2}\left(G,\lambda\right)\) implies that \(\rho\left(g^{-1}\right)\xi-\xi\in\ell^{2}\left(G\right)\). This shows that \(\mathcal{D}_{2}\left(G,\lambda\right)\subseteq\mathcal{D}_{2}\left(G,\rho\right)\). The converse follows symmetrically.
\end{proof}

For finitely generated groups, it is a basic fact that every FC-group is virtually abelian, but the converse is false: the infinite dihedral group \(D_{\infty}\) is virtually abelian, in fact virtually cyclic, but is not an FC-group. The fact that virtual commutativity does not generally obstruct \(\ell^{2}\)-asymmetry was noted by M. Kapovich \cite{MOthread}, who suggested an asymmetric \(\ell^{2}\)-cocycle on \(D_{\infty}\).

\subsection{Skewed F{\o}lner geometry}\label{sct:leftconst}

Here we present a general method for asymmetric constructions. As a convention, we present the left-handed constructions, that is, we construct elements of \(\mathcal{D}_{2}\left(G,\lambda\right)\setminus \mathcal{D}_{2}\left(G,\rho\right)\). All of our constructions have obvious right-handed versions.

\smallskip

We start by defining left schemes in a form that makes the underlying skewed F{\o}lner geometry transparent. In this way, finding left schemes becomes a simpler geometric task. For the actual asymmetric constructions we will use a further combinatorial refinement (\cref{lem:leftsch}).

\begin{defn}\label{dfn:lftsch}
Let \(G\) be a countable group and fix \(s_{\ast}\in G\). An {\bf \(s_{\ast}\)-left scheme} is a sequence \(\left(E_{n}\right)_{n\geq1}\) of nonempty finite subsets of \(G\), such that:
\begin{enumerate}
    \item \(E_{n}s_{\ast}\cap E_{n}=\emptyset\) for every \(n\geq 1\).
    \item \(\Phi_{E}\left(g\right)\coloneq\sum_{n\geq1}\frac{\left|gE_{n}\triangle E_{n}\right|}{\left|E_{n}\right|}<+\infty\) for every \(g\in G\).
\end{enumerate}
\end{defn}

\begin{rem}\label{rem:summ}
When \(G\) is generated by a symmetric set \(S\), it suffices to check that \(\Phi_{E}\left(s\right)<+\infty\) for \(s\in S\) to verify condition~(2) of left scheme. Indeed, let \(g=s_{1}\dotsm s_{k}\in G\) with \(s_{i}\in S\), \(1\leq i\leq k\). For \(A\subset G\) put \(A_{i}=s_{1}\dotsm s_{i}A\) for \(1\leq i\leq k\). Then one has the telescoping inclusion
\[A_{0}\coloneq A,\quad A_{k}=gA,\quad gA\triangle A\subseteq\bigcup\nolimits_{i=1}^{k}\left(A_{i}\triangle A_{i-1}\right)=\bigcup\nolimits_{i=1}^{k}s_{1}\dotsm s_{i-1}\left(s_{i}A\triangle A\right).\]
Since left multiplication preserves cardinalities, it follows that
\[\left|gA\triangle A\right|\leq\sum\nolimits_{i=1}^{k}\left|s_{i}A\triangle A\right|.\]
\end{rem}

We will now verify that non-FC amenable groups admit left schemes, and then we will show how left schemes give rise to asymmetric \(\ell^{2}\)-cocycles.

\begin{prop}\label{prop:amenleft}
Let \(G\) be a countable amenable group and let \(s_{\ast}\in G\) be any element. Then
\[G\text{ admits an }s_{\ast}\text{-left scheme}\quad\iff\quad s_{\ast}\text{ has an infinite conjugacy class}.\]
In particular, \(G\) admits a left scheme if and only if it is not an FC-group.
\end{prop}

\begin{proof}[Proof of \cref{prop:amenleft}]
For one implication, first note that for every finite set \(E\subseteq G\) one has
\[Es_{\ast}\triangle E\subseteq\bigcup\nolimits_{c\in s_{\ast}^{G}}\left(cE\triangle E\right);\]
indeed, if \(h\in Es_{\ast}\triangle E\), put \(c\coloneq hs_{\ast}h^{-1}\in s_{\ast}^{G}\) and we get \(h\in cE\triangle E\). Suppose now that \(\left(E_{n}\right)_{n\geq1}\) is an \(s_{\ast}\)-left scheme. Then for every \(n\geq 1\) we get
\[2\left|E_{n}\right|=\left|E_{n}s_{\ast}\triangle E_{n}\right|\leq\sum\nolimits_{c\in s_{\ast}^{G}}\left|cE_{n}\triangle E_{n}\right|,\]
where the first equality holds because \(E_{n}s_{\ast}\cap E_{n}=\emptyset\). Dividing by \(\left|E_{n}\right|\) and summing over \(n\geq 1\), we obtain
\[+\infty=\sum\nolimits_{n\geq 1}2\leq\sum\nolimits_{c\in s_{\ast}^{G}}\sum\nolimits_{n\geq 1}\frac{\left|cE_{n}\triangle E_{n}\right|}{\left|E_{n}\right|}=\sum\nolimits_{c\in s_{\ast}^{G}}\Phi_{E}\left(c\right).\]
By definition \(\Phi_{E}\left(g\right)<+\infty\) for all \(g\in G\), and therefore \(s_{\ast}^{G}\) is necessarily infinite.

\smallskip

For the converse implication, suppose \(s_{\ast}\in G\) has an infinite conjugacy class. Fix an exhaustion of \(G\) by finite sets \(K_{1}\subseteq K_{2}\subseteq\dotsm\), and using the F{\o}lner property, find a sequence \(\left(F_{n}\right)_{n\geq 1}\) of finite sets with
\[\max_{g\in K_{n}}\frac{\left|gF_{n}\triangle F_{n}\right|}{\left|F_{n}\right|}\leq\frac{1}{n^{2}}\text{ for every }n\geq 1.\]
Since \(s_{\ast}\in G\) has an infinite conjugacy class, for every \(n\geq 1\) we can choose \(g_{n}\in G\) such that
\[g_{n}s_{\ast}g_{n}^{-1}\notin F_{n}^{-1}F_{n}.\]
Define \(\left(E_{n}\right)_{n\geq 1}\) by \(E_{n}\coloneq F_{n}g_{n}\), and we verify that it is an \(s_{\ast}\)-left scheme:
\begin{enumerate}
    \item For \(n\geq 1\), if \(E_{n}s_{\ast}\cap E_{n}\neq\emptyset\) then \(hs_{\ast}=h'\) for some \(h,h'\in E_{n}=F_{n}g_{n}\), and thus
    \[g_{n}s_{\ast}g_{n}^{-1}=g_{n}h^{-1}h'g_{n}^{-1}\in F_{n}^{-1}F_{n},\]
    a contradiction to the construction. Therefore, \(E_{n}s_{\ast}\cap E_{n}=\emptyset\) for every \(n\geq 1\).
    \item Fix an arbitrary \(g\in G\). Then we have
    \[\frac{\left|gE_{n}\triangle E_{n}\right|}{\left|E_{n}\right|}=\frac{\left|gF_{n}\triangle F_{n}\right|}{\left|F_{n}\right|}\text{ for every }n\geq 1.\]
    Now fix \(N\geq1\) sufficiently large so that \(g\in K_{N}\), and we obtain
    \[\Phi_{E}\left(g\right)=\sum\nolimits_{n\geq1}\frac{\left|gE_{n}\triangle E_{n}\right|}{\left|E_{n}\right|}\leq\sum\nolimits_{n<N}\frac{\left|gE_{n}\triangle E_{n}\right|}{\left|E_{n}\right|}+\sum\nolimits_{n\geq N}\frac{1}{n^{2}}<+\infty.\qedhere\]
\end{enumerate}
\end{proof}

In the following, we construct asymmetric \(\ell^{2}\)-cocycles using left schemes as in \cref{dfn:lftsch}.

\begin{prop}\label{prop:leftschvec}
Let \(G\) be a group admitting an \(s_{\ast}\)-left scheme. Then there exists \(\xi\in\mathcal{D}_{2}\left(G,\lambda\right)\) such that \(\rho\left(s_{\ast}\right)\xi-\xi\notin\ell^{2}\left(G\right)\), and hence \(\xi\in\mathcal{D}_{2}\left(G,\lambda\right)\setminus\mathcal{D}_{2}\left(G,\rho\right)\).
\end{prop}

In order to construct the desired asymmetric \(\ell^{2}\)-cocycle, we will need to rearrange left schemes to have additional disjointness.

\begin{lem}\label{lem:leftsch}
Suppose \(G\) admits an \(s_{\ast}\)-left scheme \(\left(E_{n}\right)_{n\geq1}\). Then \(G\) admits an \(s_{\ast}\)-left scheme \(\left(D_{n}\right)_{n\geq1}\) with the following properties:
\begin{enumerate}
    \item \(D_{n}\cap D_{m}=\emptyset\) for all \(n\neq m\).
    \item \(D_{n}s_{\ast}\cap D_{m}=\emptyset\) for all \(n,m\).
    \item \(\Phi_{D}\left(g\right)=\Phi_{E}\left(g\right)<+\infty\) for every \(g\in G\).
\end{enumerate}
\end{lem}

\begin{proof}[Proof of \cref{lem:leftsch}]
Fix an \(s_{\ast}\)-left scheme \(\left(E_{n}\right)_{n\geq1}\). The first implication of \cref{prop:amenleft} holds without amenability, and shows that \(s_{\ast}\) has an infinite conjugacy class. Equivalently, using the orbit-stabilizer theorem for the conjugation action of \(G\) on itself, the centralizer
\[C_{G}\left(s_{\ast}\right)=\{g\in G:gs_{\ast}g^{-1}=s_{\ast}\}\]
has infinite index in \(G\). We will define a sequence \(\left(h_{n}\right)_{n\geq1}\subseteq G\) and put
\[D_{n}\coloneq E_{n}h_{n},\quad n\geq1.\]
Set \(h_{1}\coloneq e_{G}\). Assume that \(h_{1},\dotsc,h_{n-1}\) have already been chosen for some \(n\geq2\), and write
\[D_{m}\coloneq E_{m}h_{m}\quad\text{for }m<n.\]
We choose \(h_{n}\in G\) such that the following conditions hold:
\begin{enumerate}
    \item \(h_{n}s_{\ast}h_{n}^{-1}\notin E_{n}^{-1}E_{n}\).
    \item \(E_{n}h_{n}\cap D_{m}=\emptyset\) for every \(m<n\), i.e. \(h_{n}\notin E_{n}^{-1}D_{m}\).
    \item \(E_{n}h_{n}s_{\ast}\cap D_{m}=\emptyset\) for every \(m<n\), i.e. \(h_{n}\notin E_{n}^{-1}D_{m}s_{\ast}^{-1}\).
    \item \(D_{m}s_{\ast}\cap E_{n}h_{n}=\emptyset\) for every \(m<n\), i.e. \(h_{n}\notin E_{n}^{-1}D_{m}s_{\ast}\).
\end{enumerate}
We shall explain why such \(h_{n}\) exists. Note that for every \(a\in G\), the set
\[X_{a}\coloneq\left\{ h\in G:hs_{\ast}h^{-1}=a\right\}\]
is either empty or is a left coset of \(C_{G}\left(s_{\ast}\right)\); indeed, fixing \(h_{\ast}\in X_{a}\), then \(hs_{\ast}h^{-1}=a=h_{\ast}s_{\ast}h_{\ast}^{-1}\) is equivalent to \(\left(h_{\ast}^{-1}h\right)s_{\ast}\left(h_{\ast}^{-1}h\right)^{-1}=s_{\ast}\) and so \(h\in h_{\ast}C_{G}\left(s_{\ast}\right)\), showing that \(X_{a}=h_{\ast}C_{G}\left(s_{\ast}\right)\). Then since \(E_{n}^{-1}E_{n}\) is finite, the forbidden set in condition (1) is contained in a finite union of left cosets of \(C_{G}\left(s_{\ast}\right)\). The forbidden sets in conditions (2)-(4) are all finite, and thus trivially contained in a finite union of left cosets of \(C_{G}\left(s_{\ast}\right)\). Therefore, since \(C_{G}\left(s_{\ast}\right)\) has infinite index in \(G\), such \(h_{n}\) can indeed be found.

\smallskip

We will verify that \(\left(D_{n}\right)_{n\geq 1}\) satisfies the desired properties. The property \(D_{n}\cap D_{m}=\emptyset\) for all \(m<n\) is given by condition (2). The property \(D_{n}s_{\ast}\cap D_{m}=\emptyset\) for all \(n\neq m\) is given by conditions (3) and (4). The property \(D_{n}s_{\ast}\cap D_{n}=\emptyset\) for all \(n\) is given by condition (1): if \(D_{n}s_{\ast}\cap D_{n}\neq\emptyset\) then \(eh_{n}s_{\ast}=e'h_{n}\) for some \(e,e'\in E_{n}\), hence \(h_{n}s_{\ast}h_{n}^{-1}=e^{-1}e'\in E_{n}^{-1}E_{n}\), a contradiction to condition (1). Finally, for every \(g\in G\) and \(n\geq 1\), since \(D_{n}=E_{n}h_{n}\) and \(gD_{n}\triangle D_{n}=\left(gE_{n}\triangle E_{n}\right)h_{n}\), we have
\[\frac{\left|gD_{n}\triangle D_{n}\right|}{\left|D_{n}\right|}=\frac{\left|gE_{n}\triangle E_{n}\right|}{\left|E_{n}\right|},\quad\text{hence}\quad\Phi_{D}\left(g\right)=\Phi_{E}\left(g\right)<+\infty.\qedhere\]
\end{proof}

\begin{proof}[Proof of \cref{prop:leftschvec}]
Let \(\left(D_{n}\right)_{n\geq1}\) be an \(s_{\ast}\)-left scheme as in \cref{lem:leftsch}, and define
\[\xi\in\mathbb{R}^{G},\quad\xi\left(h\right)\coloneq\sum\nolimits_{n\geq1}\frac{1}{\sqrt{\left|D_{n}\right|}}\cdot\mathbf{1}_{D_{n}}\left(h\right),\quad h\in G.\]
By condition~(1) of \cref{lem:leftsch}, for every \(h\in G\) there is at most one \(n\) such that \(h\in D_{n}\), and hence \(\xi\left(h\right)\) is well-defined. We first show that \(\rho\left(s_{\ast}\right)\xi-\xi\notin\ell^{2}\left(G\right)\). If \(h\in D_{n}\) then by condition~(2) of \cref{lem:leftsch}, \(hs_{\ast}\notin D_{m}\) for every \(m\geq1\), and therefore
\[\left(\rho\left(s_{\ast}\right)\xi-\xi\right)\left(h\right)=\xi\left(hs_{\ast}\right)-\xi\left(h\right)=-\frac{1}{\sqrt{\left|D_{n}\right|}}.\]
It follows that
\[\left\Vert\rho\left(s_{\ast}\right)\xi-\xi\right\Vert_{\ell^{2}\left(G\right)}^{2}\geq\sum\nolimits_{n\geq1}\sum\nolimits_{h\in D_{n}}\frac{1}{\left|D_{n}\right|}=\sum\nolimits_{n\geq1}1=+\infty.\]
We now prove that \(\xi\in\mathcal{D}_{2}\left(G,\lambda\right)\). Fix \(g\in G\). Since \(\left(D_{n}\right)_{n\geq1}\) are pairwise disjoint, so are \(\left(gD_{n}\right)_{n\geq1}\). Then for every \(h\in G\), at most one of the terms \(\mathbf{1}_{D_{n}}\left(h\right)\) is nonzero, and at most one of the terms \(\mathbf{1}_{gD_{n}}\left(h\right)\) is nonzero. Therefore,
\[\Big|\sum\nolimits_{n\geq1}\frac{1}{\sqrt{\left|D_{n}\right|}}\left(\mathbf{1}_{gD_{n}}\left(h\right)-\mathbf{1}_{D_{n}}\left(h\right)\right)\Big|^{2}\leq\sum\nolimits_{n\geq1}\frac{1}{\left|D_{n}\right|}\mathbf{1}_{gD_{n}\triangle D_{n}}\left(h\right).\]
Summing over \(h\in G\), we get
\[\left\Vert\lambda\left(g\right)\xi-\xi\right\Vert_{\ell^{2}\left(G\right)}^{2}\leq\sum\nolimits_{n\geq1}\frac{\left|gD_{n}\triangle D_{n}\right|}{\left|D_{n}\right|}=\Phi_{D}\left(g\right)<+\infty.\qedhere\]
\end{proof}

\section{Recurrent left schemes and asymmetric Bernoulli constructions}\label{sct:reclftsch}

\subsection{Nonsingular actions and Bernoulli schemes}\label{sct:nonsBer}

We recall the basics of nonsingular ergodic theory needed for the Bernoulli construction.

\smallskip

Let \(G\curvearrowright\left(\Omega,\mu\right)\) be a nonsingular action of a countable discrete group on a standard probability space. This means that \(G\) acts on \(\Omega\) in a measurable way, and \(\mu\) is a probability measure whose measure class is preserved by the action of \(G\). The action is called {\bf conservative} if it is Poincar\'{e} recurrent: for every measurable set \(A\) in \(\Omega\) with \(\mu\left(A\right)>0\), there is \(e_{G}\neq g\in G\) such that \(\mu\left(g.A\cap A\right)>0\), where \(g.A\) denotes the image set of \(A\) under the action of \(g\). This is one of several equivalent forms of conservativity, and for comprehensive treatments see~\cite[\S1.3]{krengel1985}, \cite[\S1.1--1.6]{aaronson1997}, \cite{avraham2024hopf}. The action is called {\bf ergodic} if for every measurable set \(A\) in \(\Omega\) satisfying \(\mu\left(g.A\triangle A\right)=0\) for every \(g\in G\), it holds that \(\mu\left(A\right)\in\{0,1\}\). The action is called {\bf weakly mixing} if for every ergodic probability preserving action \(G\curvearrowright\left(\Omega',\mu'\right)\), the diagonal nonsingular action \(G\curvearrowright\left(\Omega\times\Omega',\mu\otimes\mu'\right)\) is ergodic. Weak mixing implies ergodicity, and when the underlying probability space is nonatomic, ergodicity implies conservativity.

\smallskip

The basic objects we are about to discuss were introduced in \cref{sct:asymBer}: Bernoulli schemes \(\left(\Omega_{G},\mu_{\xi}\right)\), \(\xi\in\left(0,1\right)^{G}\), the left and right shifts \(L\) and \(R\), and left and right nonsingularity. We will use Kakutani's theorem in the following form. By Kakutani's dichotomy~\cite{kakutani1948} (see e.g. \cite[\S1]{vaes2018} for this formulation), for every pair of vectors \(\xi,\eta\in\left(0,1\right)^{G}\), the product measures \(\mu_{\xi}\) and \(\mu_{\eta}\) are either equivalent (i.e. mutually absolutely continuous) or mutually singular, and they are equivalent if and only if both
\[\sqrt{\xi}-\sqrt{\eta}\in\ell^{2}\left(G\right)\quad\text{and}\quad\sqrt{\mathbf{1}-\xi}-\sqrt{\mathbf{1}-\eta}\in\ell^{2}\left(G\right).\]
Since \(L\left(g\right)_{\ast}\mu_{\xi}=\mu_{\lambda\left(g\right)\xi}\) and \(R\left(g\right)_{\ast}\mu_{\xi}=\mu_{\rho\left(g\right)\xi}\), this recovers the criterion for left and right nonsingularity stated in \cref{sct:asymBer}. In what follows, when \(\mu_{\xi}\) is given by \(\xi\in\left(0,1\right)^{G}\) that satisfies the left nonsingularity condition of Kakutani, we will refer to \(G\curvearrowright\left(\Omega_{G},\mu_{\xi}\right)\) as a {\bf left nonsingular Bernoulli shift}. We will use {\bf right nonsingular Bernoulli shift} analogously.

\smallskip

For amenable \(G\), every conservative left nonsingular Bernoulli shift \(\left(\Omega_{G},\mu_{\xi}\right)\) that satisfies \(0<\inf_{g\in G}\xi\left(g\right)\leq\sup_{g\in G}\xi\left(g\right)<1\) is weakly mixing. This was proved in \cite[Thm.~0.2]{danilenko2019}, generalizing the case of \(G=\mathbb{Z}\) with ergodicity proved in \cite[Thm.~3]{kosloff2019} (see also \cite[\S5]{bjorklund2021}). This result raises the significance of finding conditions on \(\xi\) ensuring that the corresponding shift is nonsingular and conservative. Here we do it using the following criterion of Vaes--Wahl, which is a special case of \cite[Prop.~4.1]{vaes2018} (see also \cite[\S5]{bjorklund2021}).

\begin{prop}[Vaes--Wahl]\label{prop:VaWa} 
Let \(G\) be a countable group, and let \(\xi\in\mathcal{D}_{2}\left(G,\lambda\right)\) satisfy \(\delta\leq\xi\left(g\right)\leq1-\delta\) for every \(g\in G\) for some \(0<\delta<1/2\). If there exists \(\kappa>\delta^{-2}+\delta^{-1}\left(1-\delta\right)^{-2}\) such that
\[\sum\nolimits_{g\in G}\exp\left(-\kappa\|\lambda(g)\xi-\xi\|_2^2\right)=+\infty,\]
then the corresponding left nonsingular Bernoulli shift \(G\curvearrowright\left(\Omega_{G},\mu_{\xi}\right)\) is conservative.
\end{prop}

\subsection{Recurrent left schemes}

The criterion of conservativity given in \cref{prop:VaWa} motivates the following refinement of \cref{dfn:lftsch} of left scheme.

\begin{defn}\label{dfn:reclftsch}
Let \(G\) be a countable group and fix \(s_{\ast}\in G\). An \(s_{\ast}\)-left scheme \(\left(E_{n}\right)_{n\geq 1}\) is called {\bf recurrent} if there exists \(\kappa>0\) such that
\[\sum\nolimits_{g\in G}\exp\left(-\kappa\Phi_{E}\left(g\right)\right)=+\infty.\]
\end{defn}

\smallskip

In the following, we use recurrent left schemes as in \cref{dfn:reclftsch} to construct weakly mixing left nonsingular Bernoulli shifts that fail to be right nonsingular.

\begin{prop}\label{prop:leftschbern}
Let \(G\) be a countable amenable group admitting a recurrent \(s_{\ast}\)-left scheme. Then there exists a \(\left(0,1\right)\)-valued vector \(\xi\in\mathcal{D}_{2}\left(G,\lambda\right)\setminus\mathcal{D}_{2}\left(G,\rho\right)\) whose corresponding Bernoulli scheme \(\left(\Omega_{G},\mu_{\xi}\right)\) satisfies the following properties:
\begin{enumerate}
    \item \(G\curvearrowright\left(\Omega_{G},\mu_{\xi}\right)\) is a conservative, weakly mixing, left nonsingular Bernoulli shift.
    \item \(G\curvearrowright\left(\Omega_{G},\mu_{\xi}\right)\) is not right nonsingular: the right shift by \(s_{\ast}\) is singular.
\end{enumerate}
\end{prop}

In order to construct the desired asymmetric Bernoulli schemes, we will verify that \cref{lem:leftsch} holds true also for recurrent left schemes:

\begin{lem}\label{lem:leftschbern}
Suppose \(G\) admits a recurrent \(s_{\ast}\)-left scheme \(\left(E_{n}\right)_{n\geq1}\). Then \(G\) admits a recurrent \(s_{\ast}\)-left scheme \(\left(D_{n}\right)_{n\geq1}\) satisfying conditions (1)--(3) of \cref{lem:leftsch}.
\end{lem}

\begin{proof}[Proof of \cref{lem:leftschbern}]
Apply \cref{lem:leftsch} to the underlying \(s_{\ast}\)-left scheme \(\left(E_{n}\right)_{n\geq1}\), and let \(\left(D_{n}\right)_{n\geq1}\) be the resulting refinement. For \(g\in G\), denote by \(\Phi_{E}\left(g\right)\) and \(\Phi_{D}\left(g\right)\) the functions associated with \(\left(E_{n}\right)_{n\geq 1}\) and \(\left(D_{n}\right)_{n\geq 1}\), and by \cref{lem:leftsch}(3) we have \(\Phi_{D}\left(g\right)=\Phi_{E}\left(g\right)<+\infty\). Then since \(\left(E_{n}\right)_{n\geq1}\) is recurrent, say with parameter \(\kappa>0\), it readily follows that \(\left(D_{n}\right)_{n\geq 1}\) is recurrent with the same parameter \(\kappa\).
\end{proof}

\begin{proof}[Proof of \cref{prop:leftschbern}]
Let \(\left(D_{n}\right)_{n\geq1}\) be the recurrent \(s_{\ast}\)-left scheme supplied by \cref{lem:leftschbern} with parameter \(\kappa>0\). Fix \(0<\epsilon<\min\left\{\sqrt{\kappa}/4,1/6\right\}\), and define a \(\left(0,1\right)\)-valued vector by
\[\xi\in\mathbb{R}^{G},\quad\xi\left(h\right)\coloneq\frac{1}{2}+\epsilon\cdot\sum\nolimits_{n\geq1}\frac{1}{\sqrt{\left|D_{n}\right|}}\cdot\mathbf{1}_{D_{n}}\left(h\right),\quad h\in G.\]
This \(\xi\) is just an affine rescaling of the vector defined in the proof of \cref{prop:leftschvec}, and therefore \(\xi\in\mathcal{D}_{2}\left(G,\lambda\right)\) while \(\rho\left(s_{\ast}\right)\xi-\xi\notin\ell^{2}\left(G\right)\). Let us explain how Kakutani's criterion gives that the resulting Bernoulli scheme \(\left(\Omega_{G},\mu_{\xi}\right)\) is left nonsingular but not right nonsingular. By condition~(1) in \cref{lem:leftsch}, for every \(h\in G\) there is at most one \(n\) with \(h\in D_{n}\), and hence by the choice of \(\epsilon\),
\[\frac{1}{2}\leq\xi\left(h\right)\leq\frac{1}{2}+\epsilon<\frac{2}{3}.\]
As was already observed by Kakutani~\cite[Cor.~1]{kakutani1948}, if \(\eta\in\mathbb{R}^{G}\) satisfies \(0<\inf\eta\leq\sup\eta<1\), then for every \(g\in G\) one has \(\lambda\left(g\right)\eta-\eta\in\ell^{2}\left(G\right)\iff\lambda\left(g\right)\sqrt{\eta}-\sqrt{\eta}\in\ell^{2}\left(G\right)\), and the same holds with \(\rho\left(g\right)\) in place of \(\lambda\left(g\right)\). Therefore, one deduces that \(\sqrt{\xi},\sqrt{\mathbf{1}-\xi}\in\mathcal{D}_{2}\left(G,\lambda\right)\) and so the left shift is nonsingular. Similarly, one deduces that \(\rho\left(s_{\ast}\right)\sqrt{\xi}-\sqrt{\xi}\notin\ell^{2}\left(G\right)\), so by Kakutani's dichotomy the right shift by \(s_{\ast}\) is singular.

\smallskip

We next prove the conservativity. Fix \(g\in G\). For every \(h\in G\) one has
\[\left(\lambda\left(g\right)\xi-\xi\right)\left(h\right)=\epsilon\cdot\sum\nolimits_{n\geq1}\frac{1}{\sqrt{\left|D_{n}\right|}}\left(\mathbf{1}_{gD_{n}}\left(h\right)-\mathbf{1}_{D_{n}}\left(h\right)\right).\]
Since \(D_{n}\), \(n\geq1\), are pairwise disjoint, so are \(gD_{n}\), \(n\geq1\). Therefore, for every \(h\in G\) at most one of  \(\mathbf{1}_{gD_{n}}\left(h\right)\) is nonzero, and at most one of \(\mathbf{1}_{D_{n}}\left(h\right)\) is nonzero. It follows that
\[\Big|\sum\nolimits_{n\geq1}\frac{1}{\sqrt{\left|D_{n}\right|}}\left(\mathbf{1}_{gD_{n}}\left(h\right)-\mathbf{1}_{D_{n}}\left(h\right)\right)\Big|^{2}\leq\sum\nolimits_{n\geq1}\frac{1}{\left|D_{n}\right|}\mathbf{1}_{gD_{n}\triangle D_{n}}\left(h\right).\]
Then since \(\Phi_{D}\left(g\right)<+\infty\), summing over \(h\in G\) gives
\[\big\Vert\lambda\left(g\right)\xi-\xi\big\Vert_{\ell^{2}\left(G\right)}^{2}\leq\epsilon^{2}\cdot\sum\nolimits_{n\geq1}\frac{\left|gD_{n}\triangle D_{n}\right|}{\left|D_{n}\right|}=\epsilon^{2}\Phi_{D}\left(g\right)<+\infty.\]
By the choice \(0<\epsilon<\sqrt{\kappa}/4\) we have \(16\epsilon^{2}<\kappa\), hence
\[16\big\Vert\lambda\left(g\right)\xi-\xi\big\Vert_{\ell^{2}\left(G\right)}^{2}\leq\kappa\Phi_{D}\left(g\right),\]
and we obtain
\[\exp\big(-16\big\Vert\lambda\left(g\right)\xi-\xi\big\Vert_{\ell^{2}\left(G\right)}^{2}\big)\geq\exp\left(-\kappa\Phi_{D}\left(g\right)\right).\]
It follows from recurrence that
\[\sum\nolimits_{g\in G}\exp\big(-16\big\Vert\lambda\left(g\right)\xi-\xi\big\Vert_{\ell^{2}\left(G\right)}^{2}\big)\geq\sum\nolimits_{g\in G}\exp\left(-\kappa\Phi_{D}\left(g\right)\right)=+\infty.\]
Put \(\delta\coloneq 1/3\) and so \(\delta\leq\inf\xi\leq\sup\xi\leq1-\delta\). Since \(\delta^{-2}+\delta^{-1}\left(1-\delta\right)^{-2}=63/4<16\), by \cref{prop:VaWa} the left nonsingular Bernoulli shift \(G\curvearrowright\left(\Omega_{G},\mu_{\xi}\right)\) is conservative and, since \(G\) is amenable, by \cite[Thm.~0.2]{danilenko2019} it is further weakly mixing.
\end{proof}

\subsection{Constructing recurrent left schemes}\label{sct:recleftconst}

The following proposition, along with \cref{prop:leftschbern}, is the main tool we need for \cref{mthm3}.

\begin{prop}\label{prop:inficc}
Let \(G\) be a countable amenable group and let \(s_{\ast}\in G\) have an infinite conjugacy class. Then \(G\) admits a recurrent \(s_{\ast}\)-left scheme.
\end{prop}

\begin{proof}[Proof of \cref{prop:inficc}]
Fix \(q>1\). Enumerate \(G=\left\{g_{1},g_{2},\dotsc\right\}\) with \(g_{1}=e_{G}\). Define
\[K_{n}\coloneq\big\{g_{1}^{\pm1},\dotsc,g_{\left\lfloor q^{n}\right\rfloor }^{\pm1}\big\},\quad n\geq 1.\]
Then \(\left(K_{n}\right)_{n\geq1}\) is an exhaustion of \(G\) by finite symmetric sets. Since \(G\) is amenable, for every \(n\geq1\) we may choose a nonempty finite set \(F_{n}\subset G\) such that
\[\max_{g\in K_{n}}\frac{\left|gF_{n}\triangle F_{n}\right|}{\left|F_{n}\right|}\leq q^{-n}.\]
For every \(n\geq 1\), using that \(F_{n}^{-1}F_{n}\) is finite while \(s_{\ast}\) has infinite conjugacy class, choose \(h_{n}\in G\) with \(h_{n}s_{\ast}h_{n}^{-1}\notin F_{n}^{-1}F_{n}\). Define \(\left(E_{n}\right)_{n\geq 1}\) by \(E_{n}\coloneq F_{n}h_{n}\), and we verify that it is a recurrent \(s_{\ast}\)-left scheme:
\begin{enumerate}
    \item For \(n\geq 1\), if \(E_{n}s_{\ast}\cap E_{n}\neq\emptyset\) then there are \(h,h'\in F_{n}\) such that \(hh_{n}s_{\ast}=h'h_{n}\), hence \(h_{n}s_{\ast}h_{n}^{-1}=h^{-1}h'\in F_{n}^{-1}F_{n}\), contrary to the choice of \(h_{n}\). Thus \(E_{n}s_{\ast}\cap E_{n}=\emptyset\).
    \item Fix \(g\in G\), and let \(N\geq1\) be such that \(g\in K_{N}\). Since \(E_{n}=F_{n}h_{n}\) and \(gE_{n}\triangle E_{n}=\left(gF_{n}\triangle F_{n}\right)h_{n}\), we have
    \[\frac{\left|gE_{n}\triangle E_{n}\right|}{\left|E_{n}\right|}=\frac{\left|gF_{n}\triangle F_{n}\right|}{\left|F_{n}\right|}\quad\text{for every }n\geq1.\]
    Therefore,
    \[\Phi_{E}\left(g\right)=\sum\nolimits_{n\geq1}\frac{\left|gE_{n}\triangle E_{n}\right|}{\left|E_{n}\right|}\leq\sum\nolimits_{n<N}\frac{\left|gE_{n}\triangle E_{n}\right|}{\left|E_{n}\right|}+\sum\nolimits_{n\geq N}q^{-n}<+\infty.\]
    \item We verify recurrence. For \(j\geq2\), by the definition of the \(K_{n}\)'s we have
    \[N_{j}\coloneq\min\left\{n\geq1:g_{j}\in K_{n}\right\}\leq\left\lceil\log_{q}j\right\rceil,\]
    and hence \(N_{j}-1\leq\log_{q}j\). We then claim that
    \[\Phi_{E}\left(g_{j}\right)\leq2\left(N_{j}-1\right)+\sum\nolimits_{n\geq N_{j}}q^{-n}\leq2\log_{q}j+\frac{q}{q-1};\]
    indeed, for \(n<N_{j}\) we used that the general term of \(\Phi_{E}\left(g\right)\) is trivially bounded by \(2\), and for \(n\geq N_{j}\), since in this case \(g_{j}\in K_{n}\), we have
    \[\frac{\left|g_{j}E_{n}\triangle E_{n}\right|}{\left|E_{n}\right|}=\frac{\left|g_{j}F_{n}\triangle F_{n}\right|}{\left|F_{n}\right|}\leq q^{-n}.\]
    Set \(\kappa\coloneq\frac{1}{2}\log q\). Then \(2\kappa\log_{q}j=\log j\), and therefore
    \[\exp\left(-\kappa\Phi_{E}\left(g_{j}\right)\right)\geq\frac{1}{j}\cdot\exp\Big(-\frac{\kappa q}{q-1}\Big)\quad\text{for every }j\geq2.\]
    We conclude that
    \[\sum\nolimits_{g\in G}\exp\left(-\kappa\Phi_{E}\left(g\right)\right)=\sum\nolimits_{j\geq1}\exp\left(-\kappa\Phi_{E}\left(g_{j}\right)\right)\geq\exp\Big(-\frac{\kappa q}{q-1}\Big)\cdot\sum\nolimits_{j\geq1}\frac{1}{j}=+\infty.\qedhere\]
\end{enumerate}
\end{proof}

\section{Geometric realizations of left schemes}\label{sct:geomleftsch}

We now give concrete geometric realizations of recurrent left schemes, where the geometry of the groups can be used to produce such constructions explicitly. This section is not needed for the proofs of the main theorems, and its purpose is to exhibit the skewed F{\o}lner geometry in two examples.

\smallskip

We start with the Heisenberg group, where the displacement comes from a drift created by commutators, and then consider wreath products \(A\wr\mathbb{Z}\) for amenable \(A\), where the displacement comes from the shift. In an earlier arXiv version of this paper, we developed these concrete constructions for nilpotent groups in greater detail using Malcev theory. We omit them here, as the general construction in \cref{prop:amenleft,prop:inficc} supersedes them.

\subsection{The Heisenberg group}\label{sct:Heis}

The Heisenberg group gives a concrete model of skewed F{\o}lner geometry. Consider the Heisenberg group
\[H_{3}\left(\mathbb{Z}\right)=\left\langle x,y,z:\left[y,x\right]=z,\left[z,x\right]=\left[z,y\right]=e_{H}\right\rangle,\]
where the commutators are with the convention \(\left[g,h\right]=g^{-1}h^{-1}gh\). Thus the multiplication is
\[x^{a}y^{b}z^{c}\cdot x^{a'}y^{b'}z^{c'}=x^{a+a'}y^{b+b'}z^{c+c'+ba'},\quad a,b,c,a',b',c'\in\mathbb{Z}.\]
We work with the finite symmetric generating set
\[S\coloneq\left\{x,x^{-1},y,y^{-1}\right\}.\]
Let us construct a recurrent \(x\)-left scheme for \(H_{3}\left(\mathbb{Z}\right)\).

\smallskip

Choose positive integer sequences \(A_{n},B_{n},C_{n}\) such that
\[\sum\nolimits_{n\geq1}\frac{1}{A_{n}}<+\infty,\quad\sum\nolimits_{n\geq1}\frac{1}{B_{n}}<+\infty,\quad\sum\nolimits_{n\geq1}\frac{A_{n}}{C_{n}}<+\infty,\quad A_{n}\geq 2^{n}.\]
For instance, \(A_{n}=2^{n}\), \(B_{n}=n^{2}\), \(C_{n}=4^{n}\). Define an integer sequence \(\left(b_{n}\right)_{n\geq1}\) inductively by
\[b_{1}\coloneq C_{1}+1,\quad b_{n+1}\coloneq\max\left\{C_{n+1}+1,\,b_{n}+B_{n}+1\right\},\]
and put
\[I_{n}\coloneq\left[b_{n},b_{n}+B_{n}\right)\cap\mathbb{Z}.\]
This choice is designed to have the properties
\[b_{n}>C_{n}\quad\text{and}\quad I_{n}\cap I_{m}=\emptyset\quad\text{for all }n\neq m.\]
Define
\[E_{n}\coloneq\left\{x^{a}y^{b}z^{c}:0\leq a<A_{n},\ b\in I_{n},\ 0\leq c<C_{n}\right\}.\]
Note that \(\left|E_{n}\right|=A_{n}B_{n}C_{n}\). We will show that \(\left(E_{n}\right)_{n\geq1}\) is a recurrent \(x\)-left scheme.

\begin{enumerate}
    \item First we verify the displacement condition. Right multiplication by \(x\) gives
    \[x^{a}y^{b}z^{c}x=x^{a+1}y^{b}z^{c+b}.\]
    For \(x^{a}y^{b}z^{c}\in E_{n}\), since \(b\geq b_{n}>C_{n}\) we have \(c+b\geq b_{n}>C_{n}\). It follows that the new \(z\)-coordinate cannot lie in \(\left[0,C_{n}\right)\), and so \(E_{n}x\cap E_{n}=\emptyset\).
    \item We now estimate left boundaries. For \(x\) we have \(\left|xE_{n}\setminus E_{n}\right|=\left|E_{n}\setminus xE_{n}\right|=B_{n}C_{n}\), so
    \[\frac{\left|xE_{n}\triangle E_{n}\right|}{\left|E_{n}\right|}=\frac{2B_{n}C_{n}}{A_{n}B_{n}C_{n}}=\frac{2}{A_{n}},\]
    and the same estimate holds for \(x^{-1}\). For \(y\) we use
    \[yx^{a}y^{b}z^{c}=x^{a}y^{b+1}z^{c+a}.\]
    Because \(\left|yE_{n}\setminus E_{n}\right|=\left|E_{n}\setminus yE_{n}\right|\) it suffices to estimate \(\left|E_{n}\setminus yE_{n}\right|\). If \(g=x^{a}y^{b}z^{c}\in E_{n}\setminus yE_{n}\) then
    \[y^{-1}g=x^{a}y^{b-1}z^{c-a}.\]
    Thus either \(b-1\notin I_{n}\) or \(c-a\notin\left[0,C_{n}\right)\). Since \(b\in I_{n}\) and \(0\leq c<C_{n}\), either \(b=b_{n}\) or \(c<a\). Therefore,
    \[\left|E_{n}\setminus yE_{n}\right|\leq A_{n}C_{n}+B_{n}\sum\nolimits_{a=0}^{A_{n}-1}\min\left\{a,C_{n}\right\}\leq A_{n}C_{n}+\frac{1}{2}A_{n}^{2}B_{n}.\]
    As \(\left|E_{n}\right|=A_{n}B_{n}C_{n}\), we get
    \[\frac{\left|yE_{n}\triangle E_{n}\right|}{\left|E_{n}\right|}\leq\frac{2}{B_{n}}+\frac{A_{n}}{C_{n}}.\]
    The same estimate holds for \(y^{-1}\). By the choice of \(A_{n},B_{n},C_{n}\), we conclude that
    \[\Phi_{E}\left(s\right)=\sum\nolimits_{n\geq1}\frac{\left|sE_{n}\triangle E_{n}\right|}{\left|E_{n}\right|}<+\infty\quad\text{for every }s\in S.\]
    In view of \cref{rem:summ}, \(\Phi_{E}\left(g\right)<+\infty\) for all \(g\in H_{3}\left(\mathbb{Z}\right)\).
    \item We verify recurrence. For \(k\in\mathbb{Z}\), left multiplication by \(x^{k}\) gives
    \[x^{k}E_{n}=\left\{x^{a}y^{b}z^{c}:k\leq a<A_{n}+k,\ b\in I_{n},\ 0\leq c<C_{n}\right\},\]
    and therefore
    \[\left|x^{k}E_{n}\triangle E_{n}\right|=2\min\left\{\left|k\right|,A_{n}\right\}B_{n}C_{n}.\]
    Using that \(\left|E_{n}\right|=A_{n}B_{n}C_{n}\), it follows that
    \[\frac{\left|x^{k}E_{n}\triangle E_{n}\right|}{\left|E_{n}\right|}=2\min\Big\{\frac{\left|k\right|}{A_{n}},1\Big\}.\]
    Since \(A_{n}\geq 2^{n}\), the same estimate as in the proof of \cref{prop:inficc} with \(q=2\) gives
    \[\Phi_{E}\left(x^{k}\right)\leq2\log_{2}\left|k\right|+4\quad\text{for every }\left|k\right|\geq1.\]
    Letting \(\kappa\coloneq\log\sqrt{2}\), we get
    \[\sum\nolimits_{g\in H_{3}\left(\mathbb{Z}\right)}\exp\left(-\kappa\Phi_{E}\left(g\right)\right)\geq\sum\nolimits_{k\geq1}\frac{1}{4k}=+\infty.\]
\end{enumerate}

\subsection{Amenable wreath products}\label{sct:wreathscheme}

Let \(A\) be a nontrivial countable amenable group (possibly finite). Consider the wreath product
\[A\wr\mathbb{Z}=H\rtimes_{\phi}\left\langle s_{\ast}\right\rangle,\quad H\coloneq\bigoplus\nolimits_{\mathbb{Z}}A,\]
where \(s_{\ast}\) has infinite order and \(\phi\) is the shift automorphism given by
\[\phi:\big(\left(a_{j}\right)_{j\in\mathbb{Z}}\big)\mapsto\left(a_{j+1}\right)_{j\in\mathbb{Z}}.\]
Every element has unique normal form \(\eta s_{\ast}^{k}\) with \(\eta\in H\), \(k\in\mathbb{Z}\). The multiplication is given by
\[\left(\eta s_{\ast}^{k}\right)\big(\eta' s_{\ast}^{k'}\big)=\eta\phi^{k}\left(\eta'\right)s_{\ast}^{k+k'}.\]
For \(a\in A\) and \(i\in\mathbb{Z}\), let \(\delta_{i}\left(a\right)\in H\) denote the configuration which is equal to \(a\) at \(i\) and to \(e_{A}\) elsewhere. With the above convention one has \(\phi^{k}\left(\delta_{i}\left(a\right)\right)=\delta_{i-k}\left(a\right)\).

\smallskip

Let \(Q_{1}\subseteq Q_{2}\subseteq\dotsm\) be an exhaustion of \(A\) by finite symmetric sets, and using the amenability of \(A\), for every \(n\geq1\) we may choose a finite nonempty set \(K_{n}\subset A\) such that
\[\max_{q\in Q_{n}}\frac{\left|qK_{n}\triangle K_{n}\right|}{\left|K_{n}\right|}\leq\frac{1}{n^{2}}.\]
Put \(L_{n}\coloneq 2^{n}\) and fix \(e_{A}\neq a_{\ast}\in A\). Define blocks with \(L_{n}\) free coordinates and two fixed lamps \(a_{\ast}\) at the endpoints:
\[R_{n}\coloneq\left\{\eta\in H:\operatorname{supp}\left(\eta\right)\subseteq\left[-1,L_{n}\right],\,\eta\left(-1\right)=\eta\left(L_{n}\right)=a_{\ast},\,\eta\left(l\right)\in K_{n}\text{ for }0\leq l<L_{n}\right\}.\]
One can verify that the sets \(\phi^{m}\left(R_{n}\right)\), \(n\geq1\), \(m\in\mathbb{Z}\), are pairwise disjoint. Define
\[E_{n}\coloneq\bigsqcup\nolimits_{l=0}^{L_{n}-1}\phi^{l}\left(R_{n}\right)s_{\ast}^{l},\quad n\geq 1,\]
and we will show that \(\left(E_{n}\right)_{n\geq1}\) forms a recurrent \(s_{\ast}\)-left scheme.
\begin{enumerate}
    \item We verify the displacement condition. Right multiplication by \(s_{\ast}\) gives
    \[E_{n}s_{\ast}=\bigsqcup\nolimits_{l=0}^{L_{n}-1}\phi^{l}\left(R_{n}\right)s_{\ast}^{l+1}.\]
    The only possible common \(s_{\ast}\)-powers of \(E_{n}s_{\ast}\) and \(E_{n}\) are \(1,\dotsc,L_{n}-1\). For \(1\leq l\leq L_{n}-1\), the \(H\)-coordinate coming from \(E_{n}s_{\ast}\) lies in \(\phi^{l-1}\left(R_{n}\right)\), while the \(H\)-coordinate coming from \(E_{n}\) lies in \(\phi^{l}\left(R_{n}\right)\). Since these sets are disjoint, \(E_{n}s_{\ast}\cap E_{n}=\emptyset\).
    \item The group \(A\wr\mathbb{Z}\) is generated by the symmetric set
    \[S\coloneq\left\{s_{\ast}^{\pm1}\right\}\cup\left\{\delta_{0}\left(a\right):a\in A\right\}.\]
    In view of \cref{rem:summ}, it suffices to show that \(\Phi_{E}\left(s\right)<+\infty\) for every \(s\in S\).
    \begin{itemize}
        \item We check \(s_{\ast}^{\pm 1}\). Left multiplication by \(s_{\ast}\) gives
        \[s_{\ast}E_{n}=\bigsqcup\nolimits_{l=0}^{L_{n}-1}\phi^{l+1}\left(R_{n}\right)s_{\ast}^{l+1}.\]
        Then the only levels in \(s_{\ast}E_{n}\triangle E_{n}\) are the first level of \(E_{n}\) and the last level of \(s_{\ast}E_{n}\). Each of these levels has cardinality \(\left|R_{n}\right|\), and since \(\left|E_{n}\right|=L_{n}\left|R_{n}\right|\) we get
        \[\frac{\left|s_{\ast}E_{n}\triangle E_{n}\right|}{\left|E_{n}\right|}=\frac{2}{L_{n}}=\frac{1}{2^{n-1}}.\]
        Hence \(\Phi_{E}\left(s_{\ast}\right)<+\infty\). The same estimate holds for \(s_{\ast}^{-1}\), so also \(\Phi_{E}\left(s_{\ast}^{-1}\right)<+\infty\).
        \item We check \(\delta_{0}\left(a\right)\) for \(a\in A\). Fix \(a\in A\). For \(0\leq l\leq L_{n}-1\), left multiplication by \(\delta_{0}\left(a\right)\) gives
        \[\delta_{0}\left(a\right)\phi^{l}\left(\eta\right)s_{\ast}^{l}=\phi^{l}\left(\phi^{-l}\left(\delta_{0}\left(a\right)\right)\eta\right)s_{\ast}^{l}=\phi^{l}\left(\delta_{l}\left(a\right)\eta\right)s_{\ast}^{l}.\]
        Then the set \(\delta_{l}\left(a\right)R_{n}\) is defined the same as \(R_{n}\) upon replacing \(K_{n}\) by \(aK_{n}\) in the \(l\)-th coordinate, while the other \(L_{n}-1\) coordinates still range over \(K_{n}\). Therefore,
        \[\left|\delta_{l}\left(a\right)R_{n}\triangle R_{n}\right|=\left|aK_{n}\triangle K_{n}\right|\left|K_{n}\right|^{L_{n}-1}.\]
        Since the different \(s_{\ast}\)-levels are disjoint, when summing over \(l=0,\dotsc,L_{n}-1\) and using that \(\left|E_{n}\right|=L_{n}\left|R_{n}\right|=L_{n}\left|K_{n}\right|^{L_{n}}\), we obtain
        \[\frac{\left|\delta_{0}\left(a\right)E_{n}\triangle E_{n}\right|}{\left|E_{n}\right|}=\frac{\left|aK_{n}\triangle K_{n}\right|}{\left|K_{n}\right|}.\]
        Choosing \(N\geq1\) such that \(a\in Q_{n}\) for all \(n\geq N\), we obtain
        \[\Phi_{E}\left(\delta_{0}\left(a\right)\right)\leq\sum\nolimits_{n<N}\frac{\left|\delta_{0}\left(a\right)E_{n}\triangle E_{n}\right|}{\left|E_{n}\right|}+\sum\nolimits_{n\geq N}\frac{1}{n^{2}}<+\infty.\]
    \end{itemize}
    \item We verify recurrence. For \(k\in\mathbb{Z}\), left multiplication by \(s_{\ast}^{k}\) gives
    \[s_{\ast}^{k}E_{n}=\bigsqcup\nolimits_{l=k}^{L_{n}+k-1}\phi^{l}\left(R_{n}\right)s_{\ast}^{l},\]
    and therefore
    \[\left|s_{\ast}^{k}E_{n}\triangle E_{n}\right|=2\min\left\{\left|k\right|,L_{n}\right\}\left|R_{n}\right|.\]
    Using that \(\left|E_{n}\right|=L_{n}\left|R_{n}\right|\), it follows that
    \[\frac{\left|s_{\ast}^{k}E_{n}\triangle E_{n}\right|}{\left|E_{n}\right|}=2\min\Big\{\frac{\left|k\right|}{L_{n}},1\Big\}.\]
    Since \(L_{n}=2^{n}\), the same estimate as in the proof of \cref{prop:inficc} with \(q=2\) gives
    \[\Phi_{E}\left(s_{\ast}^{k}\right)\leq2\log_{2}\left|k\right|+4\quad\text{for every }\left|k\right|\geq1.\]
    Letting \(\kappa\coloneq\log\sqrt{2}\), we get
    \[\sum\nolimits_{g\in A\wr\mathbb{Z}}\exp\left(-\kappa\Phi_{E}\left(g\right)\right)\geq\sum\nolimits_{k\geq1}\frac{1}{4k}=+\infty.\]
\end{enumerate}

\section{Assembly of the proofs}\label{sct:finproofs}

Let us summarize the final proofs of the main theorems.

\begin{proof}[Proof of \cref{mthm1}]
One implication is \cref{thm:nonamenable}, and the other is \cref{prop:symmetcoc}.
\end{proof}

\begin{proof}[Proof of \cref{cor:dense}]
Let \(\eta\in\mathcal{D}_{2}\left(G,\lambda\right)\) be arbitrary. If \(\eta\notin\mathcal{D}_{2}\left(G,\rho\right)\) there is nothing to prove, so assume that \(\eta\in\mathcal{D}_{2}\left(G,\rho\right)\). Fix \(\xi\in\mathcal{D}_{2}\left(G,\lambda\right)\setminus\mathcal{D}_{2}\left(G,\rho\right)\), which exists by \cref{mthm2} because \(G\) is not an FC-group, and define
\[\eta_{t}\coloneq\eta+t\xi\in\mathcal{D}_{2}\left(G,\lambda\right),\quad t>0.\]
If for some \(t>0\) we have that \(\eta_{t}\in\mathcal{D}_{2}\left(G,\rho\right)\), then \(\xi=\left(\eta_{t}-\eta\right)/t\in\mathcal{D}_{2}\left(G,\rho\right)\), contrary to the choice of \(\xi\). Therefore,
\[\eta_{t}\in\mathcal{D}_{2}\left(G,\lambda\right)\setminus\mathcal{D}_{2}\left(G,\rho\right)\quad\text{for every }t>0.\]
On the other hand, for every \(g\in G\) we have
\[\left\Vert\lambda\left(g\right)\left(\eta_{t}-\eta\right)-\left(\eta_{t}-\eta\right)\right\Vert_{\ell^{2}\left(G\right)}=t\left\Vert\lambda\left(g\right)\xi-\xi\right\Vert_{\ell^{2}\left(G\right)}\to0\quad\text{as }t\searrow0,\]
and also \(\eta_{t}\left(e_{G}\right)-\eta\left(e_{G}\right)=t\xi\left(e_{G}\right)\to0\) as \(t\searrow0\). Thus \(\eta_{t}\to\eta\) in the Hausdorff version of the Dirichlet topology as well.
\end{proof}

\begin{proof}[Proof of \cref{mthm2}]
If \(G\) is an FC-group then \(\mathcal{D}_{2}\left(G,\lambda\right)=\mathcal{D}_{2}\left(G,\rho\right)\) by \cref{prop:FC}. Conversely, assume that \(G\) is not an FC-group. Since \(G\) is countably infinite amenable, \cref{prop:amenleft} gives a left scheme on \(G\). Hence \cref{prop:leftschvec} gives \(\mathcal{D}_{2}\left(G,\lambda\right)\neq\mathcal{D}_{2}\left(G,\rho\right)\).
\end{proof}

\begin{proof}[Proof of \cref{mthm3}]
Because \(G\) is amenable and not an FC-group, by \cref{prop:inficc} it admits a recurrent left scheme. Then the desired Bernoulli scheme is supplied by \cref{prop:leftschbern}.
\end{proof}

Finally, let us explain \cref{rem:partcons}. If \(s_{\ast}\) has infinite order and infinite conjugacy class, then in the construction of \cref{prop:inficc} we may choose the sets \(K_{n}\) so that \(\left\{s_{\ast},s_{\ast}^{-1}\right\}\subseteq K_{n}\) for every \(n\geq1\). Then using a telescoping estimate as in \cref{rem:summ} gives
\[\Phi_{E}\left(s_{\ast}^{k}\right)\leq\sum\nolimits_{n\geq1}\min\left\{2,\left|k\right|q^{-n}\right\}\leq2\log_{q}\left|k\right|+\frac{q}{q-1}\quad\text{for every }\left|k\right|\geq1.\]
Thus, for \(\kappa\coloneq\frac{1}{2}\log q\) and \(\Gamma\coloneq\left\langle s_{\ast}\right\rangle\cong\mathbb{Z}\),
\[\sum\nolimits_{\gamma\in\Gamma}\exp\left(-\kappa\Phi_{E}\left(\gamma\right)\right)=+\infty.\]
After applying \cref{lem:leftschbern}, we still have \(\Phi_{D}=\Phi_{E}\), so the same estimate holds for the left scheme used to define \(\xi\) in the proof of \cref{prop:leftschbern}. Applying the conservativity argument from \cref{prop:leftschbern} to the restricted left shift \(\Gamma\curvearrowright\left(\Omega_{G},\mu_{\xi}\right)\), we get that the left shift by \(s_{\ast}\) is conservative. Note that the Vaes--Wahl criterion for conservativity applies to this setup as well (see the formulation in \cite[Prop.~4.1]{vaes2018}).

\appendix

\section{Symmetric \texorpdfstring{\(\ell^{2}\)}{l2}-cocycles on nonamenable groups}\label{app:vaes}

The following result is due to Stefaan Vaes, who kindly allowed us to include his proof.

\begin{thm}[S.~Vaes]\label{thm:nonamenable}
Let \(G\) be a countable nonamenable group and let \(\xi\in\mathbb{R}^{G}\) satisfy
\[\lambda\left(g\right)\xi-\xi\in\ell^{2}\left(G\right)\quad\text{and}\quad\rho\left(g\right)\xi-\xi\in\ell^{2}\left(G\right)\quad\text{ for every }g\in G.\]
Then \(\xi\in\ell^{2}\left(G\right)\oplus\mathbb{R}\mathbf{1}\).
\end{thm}

We recall a type of Reiter/Kesten spectral gap reformulation of nonamenability.

\begin{lem}\label{lem:noname}
If \(G\) is a nonamenable group, there is a finite set \(F\subset G\) such that the operator
\[T\coloneq\frac{1}{\left|F\right|}\sum\nolimits_{g\in F}\left(\lambda\left(g\right)-\mathrm{Id}\right)\]
(well-defined on the entire \(\mathbb{R}^{G}\)) restricts to an invertible operator on \(\ell^{2}\left(G\right)\).
\end{lem}

\begin{proof}[Proof of \cref{lem:noname}]
Since \(G\) is nonamenable, the trivial representation is not weakly contained in the left regular representation \cite[\S G.3]{bekka2008}. Thus there exist a finite set \(F\subset G\) and \(\epsilon>0\) with
\[\max_{g\in F}\left\Vert\lambda\left(g\right)\xi-\xi\right\Vert_{\ell^{2}\left(G\right)}
\geq\epsilon\quad\text{whenever }\left\Vert\xi\right\Vert_{\ell^{2}\left(G\right)}=1.\]
Replacing \(F\) by \(F\cup F^{-1}\), we may assume that \(F\) is symmetric. Set
\[U\coloneq\frac{1}{\left|F\right|}\sum\nolimits_{g\in F}\lambda\left(g\right),\text{ so that }T=U-\mathrm{Id}.\]
Since \(F\) is symmetric \(U\) is self-adjoint. For every unit vector \(\xi\in\ell^{2}\left(G\right)\) we have
\[2-2\left\langle U\xi,\xi\right\rangle=\frac{1}{\left|F\right|}\sum\nolimits_{g\in F}\left\Vert\lambda\left(g\right)\xi-\xi\right\Vert_{\ell^{2}\left(G\right)}^{2}\geq\frac{\epsilon^{2}}{\left|F\right|},\quad\text{hence}\quad\left\langle U\xi,\xi\right\rangle\leq1-\frac{\epsilon^{2}}{2\left|F\right|}.\]
Letting \(\alpha\coloneq 1-\frac{\epsilon^{2}}{2\left|F\right|}<1\), we get that \(\sup\sigma\left(U\right)\leq\alpha<1\) hence
\[\sup\sigma\left(T\right)=\sup\sigma\left(U\right)-1\leq\alpha-1<0.\]
Thus \(0\notin\sigma\left(T\right)\) so that \(T\) is invertible on \(\ell^{2}\left(G\right)\).
\end{proof}

\begin{proof}[Proof of \cref{thm:nonamenable}]
Consider the operator \(T\) on \(\mathbb{R}^{G}\) produced in \cref{lem:noname}, whose restriction to \(\ell^{2}\left(G\right)\) is invertible. Then
\[T\xi=\frac{1}{\left|F\right|}\sum\nolimits_{g\in F}\left(\lambda\left(g\right)\xi-\xi\right)\in\ell^{2}\left(G\right).\]
Using that \(T\) is invertible on \(\ell^{2}\left(G\right)\), define the vectors
\[\eta\coloneq T^{-1}\left(T\xi\right)\in\ell^{2}\left(G\right)\quad\text{and}\quad \zeta\coloneq\xi-\eta,\]
and by construction \(T\zeta=0\). Now fix an arbitrary \(h\in G\), and we have
\[\rho\left(h\right)\zeta-\zeta=\left(\rho\left(h\right)\xi-\xi\right)-\left(\rho\left(h\right)\eta-\eta\right)\in\ell^{2}\left(G\right).\]
Since \(\lambda\) and \(\rho\) commute, the operator \(T\) commutes with \(\rho\left(h\right)\), hence
\[T\left(\rho\left(h\right)\zeta-\zeta\right)=\rho\left(h\right)T\zeta-T\zeta=0.\]
As \(T\) is injective on \(\ell^{2}\left(G\right)\) we get that \(\rho\left(h\right)\zeta=\zeta\). Since \(h\) is arbitrary, \(\zeta\) is a constant vector, so write \(\zeta\left(h\right)=p\in\mathbb{R}\) for all \(h\in G\). Then \(\left\Vert\xi-p\right\Vert_{\ell^{2}\left(G\right)}=\left\Vert\xi-\zeta\right\Vert_{\ell^{2}\left(G\right)}=\left\Vert\eta\right\Vert_{\ell^{2}\left(G\right)}<+\infty\).
\end{proof}

\section*{Acknowledgments}

The authors are grateful to Stefaan Vaes for sharing his unpublished result, and for allowing us to include it in Appendix~\ref{app:vaes}. While preparing the current version of the paper, we were deeply saddened to learn of the passing of Michael (Misha) Kapovich who suggested in \cite{MOthread} the first example of an asymmetric \(\ell^{2}\)-cocycle on an amenable group. We thank the anonymous referee for a careful reading and for useful comments that improved the paper.

\bibliographystyle{acm}
\bibliography{References}

\end{document}